\documentclass[a4paper,10pt]{amsart}
\usepackage{hyperref}
\usepackage{cite}
\usepackage{amssymb}
\usepackage{amsthm}

\newcommand{\Sp}{\mathrm{Sp}}
\newcommand{\GL}{\mathrm{GL}}
\newcommand{\SL}{\mathrm{SL}}

\newcommand{\fS}{{\mathfrak S}}
\newcommand{\bS}{{\mathbb S}}

\newcommand{\sV}{\mathsf{Vect}}
\newcommand{\sR}{\mathsf{Rep}}
\DeclareMathOperator{\bch}{\mathbf{ch}}

\DeclareMathOperator{\Tr}{Tr}
\DeclareMathOperator{\Hom}{Hom}
\DeclareMathOperator{\End}{End}

\newtheorem{thm}{Theorem}
\newtheorem{prop}[thm]{Proposition}

\theoremstyle{definition}
\newtheorem{defn}{Definition}
\theoremstyle{remark}
\newtheorem*{ex}{Example}

\title{On enumeration in classical invariant theory}
\author{Bruce W. Westbury}%
\address{Department of Mathematics, University of Warwick, Coventry,
  CV4 7AL}%
\email{Bruce.Westbury@warwick.ac.uk}%
\date{21 February 2014}

\begin{document}

\begin{abstract}
In this paper we study the Frobenius characters of the invariant subspaces of the tensor powers of a representation V. The main result is a formula for these characters for a polynomial functor of V involving the characters for V. This formula uses methods of enumerative combinatorics and in particular is similar to the cycle index series of regular graphs.
\end{abstract}

\maketitle

\section{Introduction}
This paper is a contribution to tensor invariant theory. The fundamental problem is
as follows. Let $G$ be a reductive group and $V$ a finite dimensional representation.
Then for $r\ge 0$, $\otimes^rV$ has an action of the symmetric group $\fS_r$ by permuting
indices and this commutes with the diagonal action of $G$. Hence the $G$-invariant subspace
has an action of $\fS_r$. The problem is then to determine the Frobenius character of this
representation. This is known to be a hard problem. The more general problem of determining
the Frobenius characters of all isotypic subspaces is equivalent to determining the
decompositions of the Schur functors evaluated on $V$ or to giving the branching rules
for the homomorphism $G\rightarrow \GL(V)$, see \cite[Theorem III]{MR0020542}.

Fix $G$ and denote the Frobenius character of the $G$-invariant subspace of $\otimes^rV$
by $I_r(V)$. Let $P$ be a homogeneous polynomial functor of degree $k$; for example, the
$k$-th symmetric power or the $k$-th alternating power and denote the character by $\bch P$.
Then the main theorem is the following formula giving $I_r(P(V))$ in terms of $I_{rk}(V)$.
Our notation for symmetric functions follows \cite[Chapter I]{MR1354144} and \cite[Chapter 7]{MR1676282}.
These text books give the necessary background on symmetric functions.

\begin{thm}\label{thm:inv} For $r\ge 0$,
\begin{equation*}
 I_r(P(V))[X] = \big\langle h_r [ X.\bch P[Y] ], I_{rk}(V)[Y] \big\rangle_Y
\end{equation*}
\end{thm}

This formula involves the following constructions on symmetric functions; the product,
 plethysm and the scalar product; also $X$ and $Y$ denote two distinct alphabets.
This formula is only useful if the symmetric functions $I_r(V)$ are known and there are very
few cases where these are known. The main examples are
\begin{itemize}
 \item Let $V$ be the defining representation of $\SL(n)$. Then
\begin{equation*}
 I_r(V) = \begin{cases}
s_{m^n} &\text{if $r=mn$} \\
0 &\text{otherwise}
\end{cases}
\end{equation*}
\item Let $V$ be the defining representation of $\Sp(2n)$. Then
\begin{equation*}
 I_r(V) = \sum_{\substack{\lambda\vdash 2r\\
       \text{columns of even length}\\ \ell(\lambda)\le 2n}}
   s_{\lambda}
\end{equation*}
\item Let $V$ be the defining permutation representation of $\fS(n)$. Then
\begin{equation*}
 \sum_{r\ge 0} I_r(V) = H_n[H_+]
\end{equation*}
where $H_n=1+h_1+\dotsb + h_n$ and $H_+= h_1+h_2+\dotsb $.
 \item Let $V$ be the adjoint representation of $\GL(n)$. Then
\begin{equation*}
 I_r(V) = \sum_{\lambda\vdash r} s_\lambda \ast s_\lambda
\end{equation*}
where $\ast$ is the internal product of symmetric functions.
\end{itemize}

The central problem in invariant theory is usually understood to be to give generators and
defining relations for the ring of invariant polynomials on $V$. An important preliminary
problem is to determine the Hilbert series of this graded ring. 

\begin{prop} For $r\ge 0$, the dimension of the space of homogeneous invariant polynomials
of degree $r$ on $P(V)$ is
\begin{equation*}
 \big\langle h_r[\bch P], I_{rk}(V)\big\rangle
\end{equation*}
\end{prop}

\begin{ex} The dimension of the space of invariants of degree $r$ of a $k$-ary form
in $n$ variables is $0$ unless $n|rk$ and in this case is
 \begin{equation*}
   \big\langle h_r[h_k], s_{m^n}\big\rangle
 \end{equation*}
where $m=rk/n$.
\end{ex}

The fundamental construction in this paper is the following symmetric function constructed
from the two symmetric functions $F$ and $G$,
\begin{equation}\label{eq:fund}
 \left\langle H[X.F[Y]],G[Y] \right\rangle_Y
\end{equation}
These can be computed in terms of the power sum symmetric functions by
\begin{equation*}
 \left\langle H[X.F[Y]],G[Y] \right\rangle_Y = 
\sum_\lambda \frac1{z_\lambda}\left\langle p_\lambda[F],G \right\rangle p_\lambda
\end{equation*}
and can also be computed in terms of the Schur symmetric functions by
\begin{equation*}
 \left\langle H[X.F[Y]],G[Y] \right\rangle_Y =
\sum_\lambda \left\langle s_\lambda[F],G \right\rangle s_\lambda
\end{equation*}

The construction \eqref{eq:fund} is an important construction in enumerative combinatorics.
Our first example is taken from \cite{MR1575574}. This example is of historical interest
as it may be the first example of the use of plethysm. The second example is from \cite{MR0108446}
and this paper contains many other examples of the use of \eqref{eq:fund} in enumeration.
\begin{ex}
A pack of cards consists of $m$ identical sets of $n$ cards. The $nm$ cards are
dealt into $n$ hands each with $m$ cards. The hands are unordered and the cards
in each hand are also unordered. Let $f(m,n)$ be the number of possible deals.
Equivalently, $f(m,n)$ is the number of $n\times n$ matrices with entries in $\{0,1,\dotsc ,m\}$,
all row sums $n$ and all column sums $m$. Two matrices are equivalent if one can be obtained from
the other by permuting the rows.

MacMahon shows that $f(m,n)$ is the coefficient of the monomial symmetric function
asociated to the partition $m^n$ in the expansion of $h_n[h_m]$. This is equivalent to
\begin{equation*}
 f(m,n)= \big\langle h_n[h_m],h_{m}^n\big\rangle
\end{equation*}
Consider all matrices with entries in $\{0,1,\dotsc ,m\}$,
all row sums $n$ and all column sums $m$. Then $\fS_n$ acts by permuting the rows
and $f(m,n)$ is the number of orbits. The cycle index series of this action is
\begin{equation*}
 \big\langle h_n[X.h_m[Y]], h_m^n[Y] \big\rangle_Y
\end{equation*}
\end{ex}

\begin{ex} 
A graph may have multiple edges and loops. A graph is $k$-regular if each vertex
has valency $k$ where each loop contributes 2 to the valence and every other edge
contributes 1.

Then \cite[(5.8)]{MR0108446} gives the number of $k$-regular graphs
on $n$ vertices as
\begin{equation*}
 \big\langle h_n[h_k],h_{nk/2}[h_2]\big\rangle 
\end{equation*}

The cycle index series of $k$-regular graphs on a set of labelled vertices is
\begin{equation*}
 \big\langle h_n[X.h_k[Y]],h_{nk/2}[h_2[Y]]\big\rangle_Y 
\end{equation*}
\end{ex}

\section{Invariant theory}
First we review the theory of polynomial functors from \cite[Appendix A]{MR1354144}.
\subsection{Polynomial functors}
Let $\sV$ be the category of finite
dimensional vector spaces and linear maps. A \emph{polynomial functor}
is a functor $P\colon\sV\rightarrow\sV$ such that for any two finite dimensional vector spaces,
$U$ and $V$, the map $P(U,V)\colon\Hom(U,V)\rightarrow\Hom(P(U),P(V))$, $f\mapsto P(f)$, is polynomial.

The polynomial functor is \emph{homogeneous} of degree $r$ if $P(U,V)$ is homogeneous of
degree $r$ for all $U$ and $V$. Every polynomial functor is a sum of homogeneous polynomial
functors.

\begin{defn}\label{schur}
Let $A$ be a finite dimensional representation of $\fS_r$ then the associated polynomial
functor is defined on objects by $V \mapsto A\otimes_{K[\fS_r]} (\otimes^r V)$.
\end{defn}
This construction defines a functor from the category of finite dimensional
representations of $\fS_r$ to the category of homogeneous polynomial functors of
degree $r$. This functor is an equivalence where the inverse equivalence is given by polarisation,
see \cite[Appendix A,(5.4)]{MR1354144}

The basic examples are that the $r$-th symmetric power functor corresponds to the
the trivial representation, the $r$-th exterior power functor corresponds to the
the sign representation and the $r$-th tensor power functor corresponds to the
the regular representation. The polynomial functors associated to the irreducible representations
of the symmetric groups are known as Schur functors. The Schur functor associated to
the partition $\lambda$ is denoted $\bS^\lambda$.

Let $G$ be an affine algebraic group. Let $\sR(G)$ be the category of finite dimensional rational
representations. Then a polynomial functor $P$ also gives a functor $P\colon\sR(G)\rightarrow\sR(G)$.

The \emph{character} of a homogeneous polynomial functor of degree $r$ is a symmetric function
of degree $r$. The character of $P$ is denoted by $\bch (P)$ and is determined by the property
that for any invertible diagonal matrix $A$ with diagonal entries $a_1,\dotsc,a_n$ we have
 \begin{equation*}
\bch (P)(a_1,\dotsc,a_n,0,0,\dotsc) = \Tr P(A)
\end{equation*}
For example, $\bch\bS^{(n)} = h_n$, $\bch \bS^{1^n} = e_n$ and  $\bch \bS^\lambda = s_\lambda$.

Let $\rho\colon \fS_r\rightarrow \End(A)$ be a representation of $\fS_r$ and $P$ the
associated polynomial functor. Then we have
\begin{equation*}
 \bch P = \frac 1{r!}\sum_{\pi\in\fS_r} \Tr \rho(\pi) p_{\lambda(\pi)}
\end{equation*}
where $\lambda(\pi)$ is the cycle type of $\pi$ and $\Tr$ is the matrix trace.

\subsection{Invariant theory}
Denote the tensor algebra of $V$ by $T^\bullet(V)$ and put $H=1+h_1+h_2+\dotsb$.
\begin{prop}\label{prop:adj} Let $P$ be a polynomial functor and $V,W$ representations of $G$. Then
 \begin{equation*}
  \dim \Hom_G (P(V),W) = \left\langle \bch P, \bch \Hom_{\fS}( T^\bullet(V), W ) \right\rangle
 \end{equation*}
\end{prop}

\begin{proof} Assume, without loss of generality, that $P$ is homogeneous of degree $r$
and that $P$ corresponds to the representation $A$. Then, by the $\otimes$-$\Hom$
adjunction applied to the bimodule $\otimes^rV$, there is a natural isomorphism of vector spaces
\begin{equation*}
 \Hom_G( A\otimes_{\fS_r}(\otimes^r V),W) \cong \Hom_{\fS_r}(A, \Hom_G(\otimes^r V, W) )
\end{equation*}
Taking dimensions gives the proposition.
\end{proof}

\begin{thm}\label{main} Let $G$ be a reductive algebraic group, $V,W\in\sR(G)$,
and $P$ a polynomial functor. Then
\begin{equation*}
 \bch \Hom_G(T^\bullet(P(V)),W) =
\left\langle H[X.(\bch P)[Y]],\bch \Hom_G(T^\bullet(V),W)[Y] \right\rangle_Y
\end{equation*}
\end{thm}

\begin{proof} The proof consists of expanding both sides to get the same expression
in both cases.

The left hand side expands to
\begin{equation*}
 \sum_\lambda \dim \Hom_G ( \bS^\lambda(P(V)),W) . s_\lambda
\end{equation*}

The Cauchy identity is
\begin{equation*}
 H[X.Y] = \sum_\lambda s_\lambda[X].s_\lambda[Y]
\end{equation*}
and so
\begin{equation*}
  H[X.(\bch P)[Y]] = \sum_\lambda s_\lambda[X].s_\lambda[\bch P[Y]]
\end{equation*}
Hence the right hand side expands to
\begin{equation*}
 \sum_\lambda \left\langle s_\lambda[\bch P], \bch \Hom_G(T^\bullet(V),W) \right\rangle . s_\lambda
\end{equation*}
The coefficients of $s_\lambda$ in these two equations are equal by Proposition \ref{prop:adj}
applied to the polynomial functor $\bS^\lambda\circ P$.
\end{proof}

Theorem \ref{thm:inv} follows by taking $W$ to be the trivial representation.
\bibliography{inv}{}

\begin{thebibliography}{Mac95}

\bibitem[Lit47]{MR0020542}
D.~E. Littlewood.
\newblock Invariants of systems of quadrics.
\newblock {\em Proc. London Math. Soc. (2)}, 49:282--306, 1947.

\bibitem[Mac17]{MR1575574}
P.~A. MacMahon.
\newblock Combinations {D}erived from m {I}dentical {S}ets of n {D}ifferent
  {L}etters and their {C}onnexion with {G}eneral {M}agic {S}quares.
\newblock {\em Proc. London Math. Soc.}, S2-17(1):25, 1917.

\bibitem[Mac95]{MR1354144}
I.~G. Macdonald.
\newblock {\em Symmetric functions and {H}all polynomials}.
\newblock Oxford Mathematical Monographs. The Clarendon Press Oxford University
  Press, New York, second edition, 1995.
\newblock With contributions by A. Zelevinsky, Oxford Science Publications.

\bibitem[Rea59]{MR0108446}
R.~C. Read.
\newblock The enumeration of locally restricted graphs. {I}.
\newblock {\em J. London Math. Soc.}, 34:417--436, 1959.

\bibitem[Sta99]{MR1676282}
Richard~P. Stanley.
\newblock {\em Enumerative combinatorics. {V}ol. 2}, volume~62 of {\em
  Cambridge Studies in Advanced Mathematics}.
\newblock Cambridge University Press, Cambridge, 1999.
\newblock With a foreword by Gian-Carlo Rota and appendix 1 by Sergey Fomin.

\end{thebibliography}
\bibliographystyle{alpha}
\end{document}